\documentclass[11pt]{amsart}

\usepackage{amsmath}
\usepackage{mathtools}
\usepackage{amssymb}
\usepackage{mathrsfs}

\usepackage[T1]{fontenc}
\usepackage{microtype}
\usepackage{xcolor}

\usepackage[hidelinks]{hyperref}

\theoremstyle{plain}
\newtheorem{theorem}{Theorem}[section]
\newtheorem{proposition}[theorem]{Proposition}
\newtheorem{lemma}[theorem]{Lemma}
\newtheorem{corollary}[theorem]{Corollary}

\theoremstyle{definition}

\theoremstyle{remark}

\numberwithin{equation}{section}

\newcommand{\Q}{\mathbb Q}
\newcommand{\R}{\mathbb R}
\newcommand{\C}{\mathbb C}
\newcommand{\HH}{\mathbb H}

\begin{document}

\title[Base-field dependence of the Hurwitz problem]
{The Hurwitz sum-of-squares problem depends on the base field}

\author{Chi Zhang}
\address{College of Sciences, Northeastern University, Shenyang, Liaoning 110004, China}
\email{zhangchi@mail.neu.edu.cn}

\author{Haoran Zhu}
\address{Division of Mathematical Sciences, Nanyang Technological University, Singapore 637371}
\email{zhuh0031@e.ntu.edu.sg}

\subjclass[2020]{Primary 11E25; Secondary 11E04, 11E39, 14G05, 05B20}
\keywords{sum-of-squares formulae; Hurwitz problem; composition of quadratic forms; formally real fields; base-field dependence}

\date{}

\begin{abstract}
We show that the Hurwitz problem for sums of squares can depend on the base field. 
More precisely, we construct an explicit formula of type $[12,12,18]$ over every field of characteristic different from $2$ in which $-1$ is a square, whereas no such formula exists over any formally real field. In particular, a formula of this type exists over $\mathbb Q(i)$ and over $\mathbb C$, but not over $\mathbb Q$ or over $\mathbb R$.
This settles, in the negative, a longstanding conjecture of Shapiro from 1984, a conjecture of Adem from 1975, and answers a signed-formula problem raised by Shapiro in 2000.

\end{abstract}

\maketitle

\section{Introduction}

The composition of sums of squares lies at the interface of quadratic forms, topology, and algebraic geometry, and it has more recently appeared in algebraic complexity theory.  
This paper shows that, in the arbitrary-field setting, the existence of a formula of a fixed type can depend on the base field.
In particular, we settle, in the negative, Shapiro's 1984 field-independence conjecture.

\subsection{Hurwitz problem}

Let $F$ be a field of characteristic different from $2$.  
A sum-of-squares formula of type $[r,s,n]$ over $F$ is an identity
\begin{equation}\label{eq:intro-def}
 (x_1^2+\cdots+x_r^2)(y_1^2+\cdots+y_s^2)
 =
 z_1^2+\cdots+z_n^2,
\end{equation}
in which each $z_\alpha$ is an \textit{$F$-bilinear} form in $x=(x_1,\ldots,x_r)$ and $y=(y_1,\ldots,y_s)$.  
Throughout, identities such as \eqref{eq:intro-def} are understood as polynomial identities.  In particular, when $F=\C$, the square $z_\alpha^2$ is the ordinary algebraic square; no complex conjugation is involved.

The problem of deciding for which triples $[r,s,n]$ such a formula exists goes back to Hurwitz~\cite{Hurwitz1898,Hurwitz1923} in 1898, and is usually known as the \textit{Hurwitz problem} for sums of squares~\cite{ShapiroBook}.  

The classical two-, four-, and eight-square identities arise from multiplication in the complex numbers, quaternions, and octonions, and Hurwitz's theorem in the diagonal case $r=s=n$ is one of the starting points of the subject.  The later Hurwitz--Radon theory determines the optimal values in the family $[r,n,n]$ through the Hurwitz--Radon function~\cite{Radon1922,ShapiroBook}.  
The theory is also closely related to the composition of quadratic forms, studied classically by Albert~\cite{Albert1942}, and to Pfister's multiplicative forms~\cite{Pfister1965}.  
For broad accounts of the history and of the general theory of compositions of quadratic forms, we refer to Shapiro's survey and monograph \cite{ShapiroSurvey,ShapiroBook}.

Over the real numbers, formulae of type $[r,s,n]$ are equivalent to norm-preserving bilinear maps
\[
   \R^r\times \R^s \longrightarrow \R^n .
\]
Thus the existence problem is intertwined with topology: such maps lead to Hopf-type constructions, vector fields on spheres, and immersions of real projective spaces.  
Classical landmarks include the work of Stiefel~\cite{Stiefel1941} and Behrend~\cite{Behrend1939} on real and formally real fields, Adams's theorem on vector fields on spheres~\cite{Adams1962}, and the Atiyah--Bott--Shapiro theory of Clifford modules~\cite{AtiyahBottShapiro1964}. 
Important later contributions include the work of Adem~\cite{Adem1975,Adem1980,Adem1981}, Lam \cite{Lam1984,Lam1985}, Yuzvinsky~\cite{Yuzvinsky1981,Yuzvinsky1983,Yuzvinsky1984}, and Yiu~\cite{Yiu1986,Yiu1987}. 
More recently, the same problem has also appeared in algebraic complexity theory: Hrube\v{s}, Wigderson, and Yehudayoff~\cite{HrubesWigdersonYehudayoff2011} related complex sums-of-squares lengths to lower bounds for non-commutative arithmetic circuits.

The \textit{arbitrary-field} version of the Hurwitz problem took shape in the early 1980s, through work of Adem~\cite{Adem1980,Adem1981}, Shapiro~\cite{Shapiro1984ArbitraryField}, and Yuzvinsky~\cite{Yuzvinsky1983}.  
Dugger and Isaksen~\cite{DuggerIsaksen2005,DuggerIsaksen2007} later proved the Hopf condition over all fields of characteristic different from $2$, transporting classical topological obstructions into algebraic $K$-theory and motivic homotopy theory, which answered a longstanding problem of Lam~\cite{Lam1984}.  
Further refinements via \'etale homotopy and Hermitian $K$-theory were obtained by Dugger--Isaksen~\cite{DuggerIsaksen2008} and Xie \cite{Xie2014} based on \cite{Swan1985}.  
These results provide powerful \textit{necessary conditions} that are uniform in the field.

\subsection{Shapiro's field-independence conjecture}

The known obstructions are field-uniform, but the existence problem could still retain arithmetic information.  We shall say that a triple $[r,s,n]$ is \emph{admissible} over $F$ if a formula of type $[r,s,n]$ exists over $F$.  Shapiro conjectured in 1984 that admissibility of a fixed triple should be independent of the coefficient field, provided the characteristic is different from $2$ \cite[Conjecture~3.8]{ShapiroSurvey}.  This field-independence conjecture was later restated by Shapiro and Szyjewski \cite{ShapiroSzyjewski1992}, and it appears again in Shapiro's monograph as Conjecture~14.22 \cite{ShapiroBook}.

This conjecture was consistent with the systematic evidence available at the time.  
In several important regimes admissibility is already independent of the base field, notably when $r\leqslant 4$ and when $s\geqslant n-2$, as recorded in~\cite[Remark~1.3]{Xie2014}.
More generally, the main obstructions in the arbitrary-field problem are themselves field-uniform.  
Dugger and Isaksen~\cite{DuggerIsaksen2007} proved the Hopf condition over all fields of characteristic different from $2$, and Xie later strengthened this circle of results via Hermitian $K$-theory~\cite{Xie2014}.  Thus, all available obstruction-theoretic evidence pointed towards field-independence.

Lynn later recast the base-field question in algebro-geometric terms~\cite{Lynn2018}.  
To each triple $[r,s,n]$ one associates a scheme $\mathcal X_{r,s,n}$, defined over $\mathbb Z$, whose $K$-points parametrise formulae of type $[r,s,n]$ over $K$.  
Lynn showed that admissibility is invariant under extension between algebraically closed fields, and she also obtained comparison results between characteristic $0$ and sufficiently large positive characteristic.  
Thus, after passage to algebraic closures, admissibility exhibits a strong rigidity.

Our main theorem shows that this rigidity breaks down over non-algebraically closed fields. In fact, the separation already occurs in the classical towers $\Q\subset \Q(i)$ and $\R\subset \C$.

\begin{theorem}\label{thm:main}
The triple $[12,12,18]$ is admissible over every field of characteristic different from $2$ in which $-1$ is a square, but is not admissible over any formally real field. In particular, it is admissible over $\Q(i)$ and over $\C$, but not over $\Q$ or over $\R$.
\end{theorem}

\begin{proof}
Lemma~\ref{lem:complex-existence} gives admissibility of $[12,12,18]$ over $\Q(i)$. Since the array constructed in Section~\ref{sec:array} has coefficients in $\mathbb Z[i]$, the same array specialises along any ring homomorphism $\mathbb Z[i]\to K$ sending $i$ to an element $u\in K$ with $u^2=-1$. Hence $[12,12,18]$ is admissible over every field $K$ of characteristic different from $2$ in which $-1$ is a square.

Now let $F$ be formally real and suppose that $[12,12,18]$ were admissible over $F$. After extension of scalars it would remain admissible over a real closure $F^{\mathrm{rc}}$. 
Admissibility of a fixed triple is the solvability of a finite system of polynomial equations with integer coefficients in the coefficients of the bilinear forms. By Tarski's principle for real closed fields, admissibility over $F^{\mathrm{rc}}$ would imply admissibility over $\R$, contradicting Lemma~\ref{lem:real-nonexistence}. The final sentence is immediate.
\end{proof}

In particular, in Lynn's algebro-geometric language, the parameter scheme $\mathcal X_{12,12,18}$ has $\Q(i)$- and $\C$-points, but no $\Q$- or $\R$-points.

Combining Theorem~\ref{thm:main} with Proposition~\ref{prop:yiu}, we obtain the following quantitative separation.

\begin{corollary}\label{cor:length-gap}
Let $F$ be a formally real field, and let $N_F(r,s)$ denote the least $n$
such that $[r,s,n]$ is admissible over $F$. Then
\[
   N_{F(i)}(12,12)\leqslant 18 < 21 \leqslant N_{F}(12,12).
\]
In particular, the least admissible length for $(12,12)$ differs over $\C$
and over $\R$ by at least $3$.
\end{corollary}

The same example also answers a signed version of the problem considered by
Shapiro.  Let $r*s$ denote the least length of a
real sum-of-squares formula of type $[r,s,n]$, and let $r\# s$ denote
the least dimension of the target of a nonsingular real bilinear map
$\mathbb R^r\times\mathbb R^s\to\mathbb R^n$.  Shapiro observed that a
real signed formula
\[
 \left(\sum_{i=1}^r x_i^2\right)\left(\sum_{j=1}^s y_j^2\right)
 =
 u_1^2+\cdots+u_p^2-v_1^2-\cdots-v_q^2
\]
implies $r\#s\le p$, and asked whether it must imply $r*s\leqslant p$;
see \cite[Ch.~14, p.~322]{ShapiroBook}.  Corollary~\ref{cor:signed-shapiro}
below gives a negative answer. 

The same construction also disproves a diagonal exact-value conjecture due to Adem, who conjectured the
corresponding diagonal values in \cite[\S 8, Problem~(8.6)]{Adem1975}: the least length
is predicted to be $26$ for $r=11,12$, $28$ for $r=13$, and $32$
for $r=14,15,16$; see~\cite[Ch.~13, pp.~269--270; Ch.~15, Notes, p.~361]{ShapiroBook}.
Our construction gives a formula of type $[12,12,18]$ over $\C$, so the
predicted value $26$ in the case $r=12$ cannot hold over all such fields.
Moreover, setting the last $x$- and $y$-variables equal to zero gives a
formula of type $[11,11,18]$, so the $r=11$ case fails as well.

\subsection{Consequences and outlook}
Theorem~\ref{thm:main} shows that the ground field is intrinsic to the Hurwitz problem rather than a dispensable parameter.
Corollary~\ref{cor:length-gap} identifies this dependence with a concrete diagonal exact-value problem originating in Adem's work and recorded in Shapiro's monograph.
Known topological, motivic, and $K$-theoretic methods give strong necessary conditions that are uniform across fields of characteristic $\neq 2$, but our result shows that no such field-uniform obstructions can by themselves characterise admissibility over all fields. What fails is not obstruction theory, but the hope that field-uniform obstructions are also sufficient. Any complete criterion for admissibility over arbitrary fields must therefore retain arithmetic information about the base field.
The signed-formula consequence further shows that allowing even one negative square can lower the number of positive squares below the least possible length of a positive-definite real sum-of-squares formula.

The construction isolates the field dependence in a small Gram-realisation problem.  The $12\times12$ coefficient array is assembled from $4\times4$ quaternionic sign templates arranged in a $3\times3$ block pattern.  At the purely combinatorial level these templates behave like orthogonal designs.  
The field enters only through the realisability of three vectors $\mathbf A$, $\mathbf B$, $\mathbf C$, introduced in Section~\ref{sec:array}, with Gram matrix $\left(
\begin{smallmatrix}
 1&-1& 1\\
-1& 1& 1\\
 1& 1& 1
\end{smallmatrix}
\right)$. This matrix is not positive semidefinite, and hence cannot arise from three unit vectors in a real positive-definite inner product.  
Over a field in which $-1$ is a square, however, it is realised by the standard bilinear dot product. Thus the counterexample is not created by altering the sign pattern itself, but by changing the bilinear geometry in which the same sign skeleton is realised.

This suggests a concrete strategy for further searches. 
One may first construct promising quaternionic, orthogonal-design, or Hadamard-type sign skeletons, and then ask over which fields the required Gram matrix can be realised.
The example in this paper shows that such Gram-realisation problems can detect arithmetic phenomena invisible to the existing field-uniform obstruction theory.

\subsection{Organisation}

In Section~\ref{sec:preliminaries}, we record a compact coefficient
criterion for sum-of-squares formulae in vector form.  In
Section~\ref{sec:separating-type}, we construct the explicit
$[12,12,18]$ formula and verify the coefficient equations using
quaternionic sign identities; we then prove the real nonexistence result by invoking Yiu's obstruction,
and derive the signed-formula consequence.

\subsection{Acknowledgements}
C.Z. was supported by the NSFC (Grant No. 12101111). 
H.Z. thanks Nanyang Technological University for its Research Scholarship and is grateful to his supervisor for constant support and encouragement.

\section{A coefficient criterion}\label{sec:preliminaries}

We shall use the following compact form of the coefficient equations.

\begin{proposition}\label{prop:criterion}
Let $F$ be a field with $\operatorname{char}F\ne 2$, and equip
$F^n$ with the bilinear form $u\cdot v=\sum_{\alpha=1}^n u_\alpha v_\alpha$.
Let $v_{ij}\in F^n$, for $1\leqslant i\leqslant r$ and $1\leqslant j\leqslant s$, and set
\[
 B(x,y)=\sum_{i=1}^r\sum_{j=1}^s v_{ij}x_i y_j
       =\bigl(z_1(x,y),\ldots,z_n(x,y)\bigr).
\]
Then
\[
 \left(\sum_{i=1}^r x_i^2\right)
 \left(\sum_{j=1}^s y_j^2\right)
 =
 \sum_{\alpha=1}^n z_\alpha(x,y)^2
\]
if and only if
\begin{equation}\label{eq:compact-criterion}
 v_{ij}\cdot v_{k\ell}+v_{i\ell}\cdot v_{kj}
 =
 2\delta_{ik}\delta_{j\ell}
\end{equation}
for all $1\leqslant i,k\leqslant r$ and $1\leqslant j,\ell\leqslant s$.
\end{proposition}

\begin{proof}
The displayed sum-of-squares identity is
\[
 B(x,y)\cdot B(x,y)
 =
 \left(\sum_i x_i^2\right)
 \left(\sum_j y_j^2\right).
\]
Since $\operatorname{char}F\ne2$, double polarisation gives the
equivalent identity
\begin{equation}\label{eq:polarised-criterion}
 B(x,y)\cdot B(x',y')
 +
 B(x,y')\cdot B(x',y)
 =
 2
 \left(\sum_i x_i x_i'\right)
 \left(\sum_j y_j y_j'\right).
\end{equation}
Evaluating \eqref{eq:polarised-criterion} at
$x=\mathbf e_i$, $x'=\mathbf e_k$, $y=\mathbf f_j$, and
$y'=\mathbf f_\ell$, where $(\mathbf e_i)$ and $(\mathbf f_j)$
are the standard bases of $F^r$ and $F^s$, gives
\eqref{eq:compact-criterion}.  Conversely,
\eqref{eq:compact-criterion} implies \eqref{eq:polarised-criterion} by
bilinearity; setting $x'=x$ and $y'=y$, and dividing by $2$, gives
the original identity.
\end{proof}

\section{The separating type $[12,12,18]$}\label{sec:separating-type}

We now prove the explicit separation.  The positive part is an explicit formula over $\Q(i)$; the negative part follows from Yiu's nonexistence theorem for real formulae of type $[12,12,20]$.

\subsection{The array over $\Q(i)$}\label{sec:array}
Let $K=\Q(i)$, and let
\[
 V=K^{18}=K^3\oplus K^{15}
\]
with the bilinear dot product $u\cdot v=\sum_{\alpha=1}^{18}u_\alpha v_\alpha$. 
Under the standard embedding $K\subset\C$, this is not the standard Hermitian inner product.

Let $\mathbf e_1,\ldots,\mathbf e_{18}$ be the standard basis of $V$.
Set
\[
 \mathbf A=\mathbf e_1,\qquad
 \mathbf B=-\mathbf e_1+\mathbf e_2+i\mathbf e_3,\qquad
 \mathbf C=\mathbf e_1+\mathbf e_2-i\mathbf e_3,
\]
and, for $2\leqslant m\leqslant16$, put
\[
 \mathbf E_m=\mathbf e_{m+2}.
\]
The subscript $m$ is chosen to match the labels in the array below.
Then
\begin{equation}\label{eq:ABC}
 \mathbf A\cdot\mathbf A
 =
 \mathbf B\cdot\mathbf B
 =
 \mathbf C\cdot\mathbf C
 =
 1,
 \qquad
 \mathbf A\cdot\mathbf B=-1,\quad
 \mathbf A\cdot\mathbf C=1,\quad
 \mathbf B\cdot\mathbf C=1,
\end{equation}
while $\mathbf E_2,\ldots,\mathbf E_{16}$ are orthonormal and
orthogonal to $\mathbf A,\mathbf B,\mathbf C$.

We display a $12\times12$ array $(v_{ij})$ with entries in $V$. In the table, $+m$ and $-m$ denote $+\mathbf E_m$ and $-\mathbf E_m$, respectively.

\begin{equation}\label{eq:main-table}
\renewcommand{\arraystretch}{1.1}
\setlength{\arraycolsep}{2.4pt}
\left[
\begin{array}{rrrrrrrrrrrr}
\mathbf A& +2& +3& +4& +5& +6& +7& +8& +9&+10&+11&+12\\
-2& \mathbf A& -4& +3& -6& +5& +8& -7&-10& +9&-12&+11\\
-3& +4& \mathbf A& -2& -7& -8& +5& +6&+11&-12& -9&+10\\
-4& -3& +2& \mathbf A& -8& +7& -6& +5&-12&-11&+10& +9\\
+5& -6& -7& -8& \mathbf B& +2& +3& +4&+13&+14&+15&+16\\
+6& +5& -8& +7& -2& \mathbf B& -4& +3&-14&+13&-16&+15\\
+7& +8& +5& -6& -3& +4& \mathbf B& -2&+15&-16&-13&+14\\
+8& -7& +6& +5& -4& -3& +2& \mathbf B&-16&-15&+14&+13\\
-9&+10&-11&+12&-13&+14&-15&+16& \mathbf C& -2& +3& -4\\
-10& -9&+12&+11&-14&-13&+16&+15& +2& \mathbf C& -4& -3\\
-11&+12& +9&-10&-15&+16&+13&-14& -3& +4& \mathbf C& -2\\
-12&-11&-10& -9&-16&-15&-14&-13& +4& +3& +2& \mathbf C
\end{array}
\right].
\end{equation}

For $\alpha=1,\dots,18$ define
\begin{equation}\label{eq:def-zalpha}
 z_\alpha(x,y)=\sum_{i=1}^{12}\sum_{j=1}^{12}(v_{ij})_\alpha x_i y_j.
\end{equation}
In other words, the coefficient of $x_i y_j$ in $z_\alpha$ is the $\alpha$-th coordinate of the vector $v_{ij}$.

To make the construction concrete, the first three bilinear forms are
\begin{align*}
 z_1 &= \sum_{t=1}^{4}x_t y_t
       -\sum_{t=5}^{8}x_t y_t
       +\sum_{t=9}^{12}x_t y_t,\\
 z_2 &= \sum_{t=5}^{12}x_t y_t,\\
 z_3 &= i\sum_{t=5}^{8}x_t y_t
       -i\sum_{t=9}^{12}x_t y_t.
\end{align*}
For $2\leqslant m\leqslant16$, the form $z_{m+2}$ is obtained by summing the
monomials $x_i y_j$ at positions labelled $+m$, and subtracting
those at positions labelled $-m$.  For example,
\begin{align*}
 z_4={}&x_1y_2-x_2y_1-x_3y_4+x_4y_3
        +x_5y_6-x_6y_5-x_7y_8+x_8y_7 \\
      &-x_9y_{10}+x_{10}y_9-x_{11}y_{12}+x_{12}y_{11}.
\end{align*}

\subsection{Verification of the coefficient equations}
\label{sec:verification}

We verify \eqref{eq:compact-criterion} for the array
\eqref{eq:main-table}.  The verification keeps the $4\times4$ sign
patterns explicit, and then explains the identities among them by means
of quaternionic multiplication.

Let
\[
 q_0=1,\qquad q_1=\mathbf i,\qquad q_2=\mathbf j,\qquad q_3=\mathbf k
\]
be the standard quaternionic basis.  The bold symbols are quaternionic
units and should not be confused with the scalar $i\in K=\Q(i)$.  We
also put
\[
 \tau(q)=\mathbf j q(-\mathbf j),
\]
so that
\[
 \tau(q_0)=q_0,\qquad
 \tau(q_1)=-q_1,\qquad
 \tau(q_2)=q_2,\qquad
 \tau(q_3)=-q_3.
\]

We shall use the following five basic sign templates, together with
the transpose of $\mathsf P$:
\begin{equation}\label{eq:explicit-sign-templates}
\setlength{\arraycolsep}{3pt}
\begin{aligned}
\mathsf D&=
\begin{pmatrix}
 q_0& q_1& q_2& q_3\\
-q_1& q_0&-q_3& q_2\\
-q_2& q_3& q_0&-q_1\\
-q_3&-q_2& q_1& q_0
\end{pmatrix},
&
\mathsf P&=
\begin{pmatrix}
 q_0& q_1& q_2& q_3\\
-q_1& q_0& q_3&-q_2\\
-q_2&-q_3& q_0& q_1\\
-q_3& q_2&-q_1& q_0
\end{pmatrix},
\\[3mm]
\mathsf Y&=
\begin{pmatrix}
 q_0& q_1& q_2& q_3\\
-q_1& q_0&-q_3& q_2\\
 q_2&-q_3&-q_0& q_1\\
-q_3&-q_2& q_1& q_0
\end{pmatrix},
&
\mathsf Z&=
\begin{pmatrix}
-q_0& q_1&-q_2& q_3\\
-q_1&-q_0& q_3& q_2\\
-q_2& q_3& q_0&-q_1\\
-q_3&-q_2&-q_1&-q_0
\end{pmatrix},
\\[3mm]
\mathsf S&=
\begin{pmatrix}
 q_0&-q_1& q_2&-q_3\\
 q_1& q_0&-q_3&-q_2\\
-q_2& q_3& q_0&-q_1\\
 q_3& q_2& q_1& q_0
\end{pmatrix}.&
\end{aligned}
\end{equation}
These matrices are sign templates.  If
$U=(u_0,u_1,u_2,u_3)$ is a four-tuple of vectors, then
$\mathsf M(U)$ is obtained from the template $\mathsf M$ by replacing
each $q_t$ by $u_t$, keeping the displayed signs.  We also use
$\mathsf P(U)^{\mathsf t}$ for the transpose of $\mathsf P(U)$.

With
\[
\begin{aligned}
 U_0&=(\mathbf A,\mathbf E_2,\mathbf E_3,\mathbf E_4),
&
 U_B&=(\mathbf B,\mathbf E_2,\mathbf E_3,\mathbf E_4),
&
 U_C&=(\mathbf C,\mathbf E_2,\mathbf E_3,\mathbf E_4),\\
 U_{12}&=(\mathbf E_5,\mathbf E_6,\mathbf E_7,\mathbf E_8),
&
 U_{13}&=(\mathbf E_9,\mathbf E_{10},\mathbf E_{11},\mathbf E_{12}),
&
 U_{23}&=(\mathbf E_{13},\mathbf E_{14},\mathbf E_{15},\mathbf E_{16}),
\end{aligned}
\]
the array \eqref{eq:main-table} is the block matrix
\begin{equation}\label{eq:block-array-direct}
\mathcal V=
\begin{pmatrix}
 \mathsf D(U_0)                 & \mathsf P(U_{12})      & \mathsf Y(U_{13})\\
 \mathsf P(U_{12})^{\mathsf t}  & \mathsf D(U_B)         & \mathsf Y(U_{23})\\
 \mathsf Z(U_{13})              & \mathsf Z(U_{23})      & \mathsf S(U_C)
\end{pmatrix}.
\end{equation}
The entries of this block matrix are exactly the vectors $v_{ij}$
displayed in \eqref{eq:main-table}.

The reason for the notation in \eqref{eq:explicit-sign-templates} is the
following uniform quaternionic description of the entries:
\begin{equation}\label{eq:quat-sign-templates}
\begin{array}{lll}
 \mathsf D_{ab}=\overline{q_a}q_b,
&
 \mathsf P_{ab}=q_b\overline{q_a},
&
 (\mathsf P^{\mathsf t})_{ab}=q_a\overline{q_b},
\\[2mm]
 \mathsf S_{ab}=\tau(q_b\overline{q_a}),
&
 \mathsf Y_{ab}=\tau(q_a)q_b,
&
 \mathsf Z_{ab}=-\tau(q_b)q_a,
\end{array}
\qquad 0\leqslant a,b\leqslant3 .
\end{equation}
Here the row and column indices of each $4\times4$ template are
$0,1,2,3$.  Thus $\mathsf M_{ab}$ denotes the signed quaternionic
basis element in the $(a,b)$-entry of the template $\mathsf M$.

For the verification we use two bilinear forms on $\HH$:
\[
 (u,v)_+=\operatorname{Re}(u\overline v),
 \qquad
 (u,v)_-=-\operatorname{Re}(uv).
\]
In the basis $q_0,q_1,q_2,q_3$, their Gram matrices are
$I_4$ and $\operatorname{diag}(-1,1,1,1)$, respectively.

\begin{lemma}\label{lem:quat-block-identities}
For $0\leqslant a,b,c,d\leqslant3$, one has
\begin{align}
 (\mathsf M_{ab},\mathsf M_{cd})_+
 +(\mathsf M_{ad},\mathsf M_{cb})_+
 &=
 2\delta_{ac}\delta_{bd}
 &&(\mathsf M\in\{\mathsf D,\mathsf P,\mathsf P^{\mathsf t},    \mathsf S,\mathsf Y,\mathsf Z\}),
 \label{eq:quat-identity-same-block}\\
 (\mathsf D_{ab},\mathsf D_{cd})_-
 +(\mathsf P_{ad},\mathsf P_{bc})_+
 &=0,
 \label{eq:quat-identity-DP}\\
 (\mathsf D_{ab},\mathsf S_{cd})_+
 +(\mathsf Y_{ad},\mathsf Z_{cb})_+
 &=0.
 \label{eq:quat-identity-DSYZ}
\end{align}
\end{lemma}
\begin{proof}
Write $N(u)=(u,u)_+$.  By construction, the templates in $\{\mathsf D,\mathsf P,\mathsf P^{\mathsf t},    \mathsf S,\mathsf Y,\mathsf Z\}$
are the coefficient arrays of the six bilinear maps
\[
\begin{array}{lll}
 \mu_{\mathsf D}(x,y)=\overline x y,
&
 \mu_{\mathsf P}(x,y)=y\overline x,
&
 \mu_{\mathsf P^{\mathsf t}}(x,y)=x\overline y,
\\[1mm]
 \mu_{\mathsf S}(x,y)=\tau(y\overline x),
&
 \mu_{\mathsf Y}(x,y)=\tau(x)y,
&
 \mu_{\mathsf Z}(x,y)=-\tau(y)x .
\end{array}
\]
Each of these maps preserves the product norm:
\[
 N(\mu_{\mathsf M}(x,y))=N(x)N(y).
\]
Polarising this identity in both variables gives
\[
 (\mu_{\mathsf M}(x,y),\mu_{\mathsf M}(z,w))_+
 +
 (\mu_{\mathsf M}(x,w),\mu_{\mathsf M}(z,y))_+
 =
 2(x,z)_+(y,w)_+ .
\]
Taking $x=q_a$, $y=q_b$, $z=q_c$, and $w=q_d$, and using that
$q_0,\ldots,q_3$ is orthonormal for $(\, ,\,)_+$, gives
\eqref{eq:quat-identity-same-block}.

For \eqref{eq:quat-identity-DP}, cyclic invariance of the real part gives
\[
 (\mathsf P_{ad},\mathsf P_{bc})_+
 =
 \operatorname{Re}(q_d\overline{q_a}q_b\overline{q_c})
 =
 \operatorname{Re}(\overline{q_a}q_b\overline{q_c}q_d)
 =
 -(\mathsf D_{ab},\mathsf D_{cd})_- .
\]

It remains to prove \eqref{eq:quat-identity-DSYZ}.  We use the elementary
identity
\begin{equation}\label{eq:twisted-quaternion-identity}
 (\overline x y,\tau(w\overline z))_+
 =
 (\tau(x)w,\tau(y)z)_+,
 \qquad x,y,z,w\in\mathbb H .
\end{equation}
Indeed, since $\tau$ is an involutive algebra automorphism preserving
$(\, ,\,)_+$, we have
\begin{align*}
 (\overline x y,\tau(w\overline z))_+
 &=
 (\tau(\overline x)\tau(y),w\overline z)_+                                      \\
 &=
 \operatorname{Re}\bigl(\tau(\overline x)\tau(y)z\overline w\bigr)              \\
 &=
 \operatorname{Re}\bigl(\tau(x)w\overline z\,\tau(\overline y)\bigr)             \\
 &=
 (\tau(x)w,\tau(y)z)_+ .
\end{align*}
Here the third equality uses $\operatorname{Re}(r)=\operatorname{Re}(\overline r)$
and cyclic invariance of the real part.

Applying \eqref{eq:twisted-quaternion-identity} with $x=q_a$, $y=q_b$, $z=q_c$, and $w=q_d$, we get
\[
 (\mathsf D_{ab},\mathsf S_{cd})_+
 =
 (\mathsf Y_{ad},-\mathsf Z_{cb})_+
 =
 -(\mathsf Y_{ad},\mathsf Z_{cb})_+.
\]
This is \eqref{eq:quat-identity-DSYZ}.
\end{proof}

\begin{proposition}
\label{prop:block-array-verification}
The entries $v_{ij}$ of the array \eqref{eq:main-table} satisfy
\begin{equation}\label{eq:block-array-compact}
 v_{ij}\cdot v_{k\ell}
 +
 v_{i\ell}\cdot v_{kj}
 =
 2\delta_{ik}\delta_{j\ell}
\end{equation}
for all $1\leqslant i,k\leqslant12$ and $1\leqslant j,\ell\leqslant12$.
\end{proposition}

\begin{proof}
Write a row or column index as
\[
 (p,a),\qquad p\in\{1,2,3\},\quad a\in\{0,1,2,3\},
\]
where $(p,a)$ corresponds to the ordinary index $4(p-1)+a+1$.
Let $T_{pq}$ denote the $(p,q)$-block of $\mathcal V$.  It is
enough to prove
\begin{equation}\label{eq:block-array-block-indices}
 T_{pq}(a,b)\cdot T_{p'q'}(c,d)
 +
 T_{pq'}(a,d)\cdot T_{p'q}(c,b)
 =
 2\delta_{pp'}\delta_{ac}\delta_{qq'}\delta_{bd}.
\end{equation}

Set
\begin{align*}
 W_0&=\operatorname{span}\{\mathbf A,\mathbf B,\mathbf C,
 \mathbf E_2,\mathbf E_3,\mathbf E_4\},&
 W_{12}&=\operatorname{span}\{\mathbf E_5,\mathbf E_6,\mathbf E_7,\mathbf E_8\},\\
 W_{13}&=\operatorname{span}\{\mathbf E_9,\mathbf E_{10},
 \mathbf E_{11},\mathbf E_{12}\},&
 W_{23}&=\operatorname{span}\{\mathbf E_{13},\mathbf E_{14},
 \mathbf E_{15},\mathbf E_{16}\}.
\end{align*}
These four spaces are mutually orthogonal.  Moreover, the block supports
are arranged as
\[
\begin{pmatrix}
 W_0    & W_{12} & W_{13}\\
 W_{12} & W_0    & W_{23}\\
 W_{13} & W_{23} & W_0
\end{pmatrix}.
\]
Thus a scalar product between entries of two blocks can be non-zero only
when the corresponding support labels in this matrix agree.

Suppose first that $p=p'$ and $q=q'$.  Then both scalar products in
\eqref{eq:block-array-block-indices} occur inside a single $4\times4$
block.  The four vectors defining each block are orthonormal for the
ambient dot product, and hence \eqref{eq:block-array-block-indices}
follows from \eqref{eq:quat-identity-same-block}.

If exactly one of the equalities $p=p'$ and $q=q'$ holds, then the two
scalar products in \eqref{eq:block-array-block-indices} pair vectors from
orthogonal support spaces.  Hence the left-hand side is zero, and so is
the right-hand side.

It remains to treat the case $p\ne p'$ and $q\ne q'$.  The right-hand
side of \eqref{eq:block-array-block-indices} is then zero.  By symmetry
in $(p,a)$ and $(p',c)$, we may assume $p<p'$.  If $\{q,q'\}\ne\{p,p'\}$, 
then the two scalar products again involve orthogonal support spaces.
Thus only the two orientations $(q,q')=(p,p')$ and $(q,q')=(p',p)$ on a principal $2\times2$ block face remain.

Consider first the face $\{1,2\}$.  The pairing between the two diagonal
blocks is represented by $(\, ,\,)_-$, because
\[
 \mathbf A\cdot\mathbf B=-1,\qquad
 \mathbf E_m\cdot\mathbf E_m=1\quad (m=2,3,4).
\]
For the orientation $(q,q')=(1,2)$, the left-hand side of
\eqref{eq:block-array-block-indices} is therefore
\[
 (\mathsf D_{ab},\mathsf D_{cd})_-
 +
 (\mathsf P_{ad},\mathsf P_{bc})_+,
\]
which vanishes by \eqref{eq:quat-identity-DP}.  The opposite orientation
$(q,q')=(2,1)$ is the same identity with $b$ and $d$ interchanged.

Now consider the faces $\{1,3\}$ and $\{2,3\}$.  In both cases the
pairing between the two diagonal blocks is represented by
$(\, ,\,)_+$, since
\[
 \mathbf A\cdot\mathbf C=1,\qquad
 \mathbf B\cdot\mathbf C=1.
\]
For the orientation $(q,q')=(p,p')$, the left-hand side of
\eqref{eq:block-array-block-indices} becomes
\[
 (\mathsf D_{ab},\mathsf S_{cd})_+
 +
 (\mathsf Y_{ad},\mathsf Z_{cb})_+,
\]
which is zero by \eqref{eq:quat-identity-DSYZ}.  The opposite orientation
$(q,q')=(p',p)$ is again obtained by interchanging $b$ and $d$.

This proves \eqref{eq:block-array-compact}.
\end{proof}

\begin{lemma}\label{lem:complex-existence}
There exists a sum-of-squares formula of type $[12,12,18]$ over
$\Q(i)$, and hence over $\C$.
\end{lemma}

\begin{proof}
By Proposition~\ref{prop:block-array-verification}, the array
$(v_{ij})$ satisfies the compact coefficient equations
\eqref{eq:compact-criterion}.  Proposition~\ref{prop:criterion} applied
to the bilinear forms \eqref{eq:def-zalpha} therefore gives a
sum-of-squares formula of type $[12,12,18]$ over $K=\Q(i)$.
Extending scalars gives the formula over $\C$.
\end{proof}

\subsection{Nonexistence over $\mathbb R$}\label{sec:nonexistence}

We now turn to the real case.  

\begin{proposition}[{\cite[Theorem~8.2]{Yiu1986}}]\label{prop:yiu}
There is no real sum-of-squares formula of type $[12,12,20]$.
\end{proposition}

\begin{lemma}\label{lem:real-nonexistence}
There is no real sum-of-squares formula of type $[12,12,n]$ for any
$n\leqslant20$. In particular, there is no real formula of type
$[12,12,18]$.
\end{lemma}

\begin{proof}
If a real formula of type $[12,12,n]$ existed for some $n\leqslant20$,
then adjoining $20-n$ zero bilinear forms would give a real formula of
type $[12,12,20]$, contradicting Proposition~\ref{prop:yiu}.
\end{proof}

\begin{corollary}\label{cor:signed-shapiro}
There exists a real signed formula
\[
 \left(\sum_{i=1}^{12}x_i^2\right)
 \left(\sum_{j=1}^{12}y_j^2\right)
 =
 u_1^2+\cdots+u_{17}^2-w^2,
\]
where $u_1,\ldots,u_{17},w$ are real bilinear forms.  However,
there is no real sum-of-squares formula of type $[12,12,n]$ for any
$n\leqslant20$.  In Shapiro's notation, this gives
\[
   12\#12\leqslant 17
   \qquad\text{but}\qquad
   12*12\geqslant 21>17.
\]
Thus a real signed formula with $p$ positive squares need not imply
$r*s\leqslant p$.
\end{corollary}

\begin{proof}
Put
\[
   w=\sum_{t=5}^{8}x_t y_t-\sum_{t=9}^{12}x_t y_t .
\]
In the explicit formula over $\Q(i)$, the forms
$z_1,z_2,z_4,\ldots,z_{18}$ have real coefficients, while
$z_3=iw$.  Hence Lemma~\ref{lem:complex-existence} gives
\[
 \left(\sum_{i=1}^{12}x_i^2\right)
 \left(\sum_{j=1}^{12}y_j^2\right)
 =
 z_1^2+z_2^2+\sum_{\alpha=4}^{18}z_\alpha^2-w^2 .
\]
This is the required signed formula with $17$ positive squares and one
negative square.

By Lemma~\ref{lem:real-nonexistence}, no real sum-of-squares formula of
type $[12,12,n]$ exists for $n\leqslant20$, so
$12*12=N_{\mathbb R}(12,12)\geqslant 21$.  On the other hand, the $17$
positive forms define a nonsingular bilinear map: if they vanished at a
pair $x\ne0$, $y\ne0$, then the displayed signed identity would give
a positive left-hand side equal to $-w(x,y)^2\leqslant0$, a contradiction.
Thus $12\#12\leqslant17$.
\end{proof}




\end{document}